\documentclass[10pt]{amsart}
\usepackage{amssymb,MnSymbol}
\usepackage{amsthm,amsmath}
\usepackage{cite}

\title[A classification of polynomial functions]{A classification of polynomial functions satisfying the Jacobi identity over integral domains}

\author{Jean-Luc Marichal}
\address{Mathematics Research Unit, University of Luxembourg, Maison du Nombre, 6, avenue de la Fonte, L-4364 Esch-sur-Alzette, Luxembourg}
\email{jean-luc.marichal[at]uni.lu}

\author{Pierre Mathonet}
\address{University of Li\`ege, Department of Mathematics, All\'ee de la D\'ecouverte 12 - B37, B-4000 Li\`ege, Belgium}
\email{p.mathonet[at]ulg.ac.be}

\date{March 10, 2017}

\theoremstyle{plain}
\newtheorem{theorem}{Theorem}

\newtheorem{proposition}[theorem]{Proposition}
\newtheorem{corollary}[theorem]{Corollary}

\newtheorem*{main}{Main Theorem}

\theoremstyle{definition}

\theoremstyle{remark}

\newtheorem*{remark}{Remark}
\newtheorem*{note}{Note}
\newtheorem{claim}{Claim}

\newcommand{\calR}{\mathcal{R}}
\newcommand{\calS}{\mathcal{S}}
\newcommand{\chR}{\mathrm{char}(\calR)}
\newcommand{\modp}{~(\mathrm{mod}~p)}

\begin{document}
\begin{abstract}
The Jacobi identity is one of the properties that are used to define the concept of Lie algebra and in this context is closely related to associativity. In this paper we provide a complete description of all bivariate polynomials that satisfy the Jacobi identity over infinite integral domains. Although this description depends on the characteristic of the domain, it turns out that all these polynomials are of degree at most one in each indeterminate.
\end{abstract}

\keywords{Jacobi's identity, polynomial, integral domain.}

\subjclass[2010]{Primary 39B72; Secondary 13B25, 17B99.}
\maketitle

\section{Introduction}

Let $\calR$ be an infinite integral domain with identity. In this paper we are interested in a classification of all bivariate polynomials $P$ over $\calR$ satisfying Jacobi's identity
\begin{equation}\label{eq:Jacobi}
P(P(x,y),z)+P(P(y,z),x)+P(P(z,x),y) ~=~ 0.
\end{equation}

To give a simple example, consider the set $\calR=\mathbb{Z}_3[x]$ of univariate polynomials whose coefficients are in $\mathbb{Z}_3$. One can easily verify that the bivariate polynomial $P$ over $\mathbb{Z}_3[x]$ defined by
$$
P(A,B) ~=~ (1-x^2)AB+(x+1)(1-x^2)(A+B)+x(x+1)(1-x-x^2)
$$
satisfies Jacobi's identity \eqref{eq:Jacobi}.

As it is well known, the Jacobi identity is one of the defining properties of Lie algebras. Recall that a \emph{Lie algebra} (see, e.g., \cite{Gil74,Hal03,Jac79}) is a vector space $\mathfrak{g}$ together with a binary map $[\cdot{\,},\cdot]\colon\mathfrak{g}\times\mathfrak{g}\to\mathfrak{g}$, called \emph{Lie bracket}, such that
\begin{enumerate}
\item[1.] $[\cdot{\,},\cdot]$ is bilinear,
\item[2.] $[x,y]=-[y,x]$ for all $x,y\in\mathfrak{g}$.
\item[3.] $[[x,y],z]+[[y,z],x]+[[z,x],y]=0$ for all $x,y,z\in\mathfrak{g}$.
\end{enumerate}
The second condition is usually called skew-symmetry while the third one is known as the Jacobi identity. By using a prefix notation for the Lie bracket, the Jacobi identity simply becomes the functional equation given in \eqref{eq:Jacobi}.

The classical associativity property is closely connected to Lie algebras in the following way (see, e.g., \cite[p.~6]{Jac79}). The Lie bracket defined by $[x,y]=xy-yx$ on any associative algebra satisfies the three properties above, including Jacobi's identity. This is one of the reasons why ``the Jacobi identity can be viewed as a substitute for associativity'' \cite[p.~54]{Hal03}.

We now state our main result, which provides a complete description of the possible polynomial solutions over $\calR$ of Jacobi's identity \eqref{eq:Jacobi}. Although the form of these polynomials depends on the characteristic of $\calR$, they are all of degree at most one in each indeterminate.

\begin{main}
Consider a bivariate polynomial $P\in\calR[x,y]$.
\begin{itemize}
\item If $\chR\neq 3$, then $P$ satisfies Jacobi's identity iff there exist $B,C\in\calR$ satisfying $B^2+BC+C=0$ such that
\[
P(x,y) ~=~ Bx+Cy,
\]
\item If $\chR=3$, then $P$ satisfies Jacobi's identity iff one of the following conditions holds:
\begin{itemize}
\item there exist $A,B,D\in\calR$ satisfying $AD=B^2-B$ such that
\[
P(x,y) ~=~ Axy+B(x+y)+D,
\]
\item there exist $B,C,D\in\calR$ satisfying $B^2+BC+C=0$ such that
\[
P(x,y) ~=~ Bx+Cy+D.
\]
\end{itemize}
\end{itemize}
\end{main}

The reader interested in possible generalizations of the Main Theorem might want to consider extensions of functional equation \eqref{eq:Jacobi} to $n$-indeterminate polynomials, by analogy with $n$-ary generalizations of Lie algebras, where Jacobi's identity involves an $n$-linear bracket. In this direction we remark that a complete classification of $n$-ary associative polynomials over $\calR$ can be found in \cite{MarMat11} and that, in the special case when $\calR$ is the complex plane $\mathbb{C}$, this classification was recently generalized to $n$-ary associative formal power series in \cite{Fri16}.

\begin{remark}
In the literature on Lie algebras the Jacobi identity is sometimes given in one of the following alternative forms (which are equivalent to the one above under bilinearity and skew-symmetry):
\begin{eqnarray}
& [x,[y,z]]+[y,[z,x]]+[z,[x,y]] ~=~ 0, &\label{eq:J2}\\
& [x,[y,z]] ~=~ [[x,y],z]+[y,[x,z]], &\label{eq:J3}\\
& [[x,y],z] ~=~ [x,[y,z]]+[[x,z],y]. &\label{eq:J4}
\end{eqnarray}
It is however easy to see that $P$ satisfies the functional equation corresponding to \eqref{eq:J2} iff the polynomial $P'$ defined by $P'(x,y)=P(y,x)$ satisfies \eqref{eq:Jacobi}. As far as equations \eqref{eq:J3} and \eqref{eq:J4} are concerned, one can show that the corresponding functional equations have no nonzero solution. The proof of this latter observation is given in Appendix~\ref{app:B}.
\end{remark}

\begin{note}
The problem addressed in this paper was suggested by J\"org Tomaschek \cite{Toma}, who in turn was asked this question by Wolfgang Prager (University of Graz, Austria) while the latter was studying local analytic solutions of the Bokov functional equation that appears in theoretical physics (see \cite{Bok73,Yag82}).
\end{note}

\section{Technicalities and proof of the Main Theorem}

We use the following notation throughout this paper. For any integer $m\geq 1$ and any prime $p\geq 2$, we denote by $s_m(p)$ the set of positive integers expressible as sums of $m$ powers of $p$, that is, integers $n$ whose base $p$ expansions $n=\sum_{i=0}^k n_i{\,}p^i$ (with $0\leq n_i<p$ for $i=0,\ldots,k$) satisfy $\sum_{i=0}^k n_i=m$. We also use the Kronecker delta symbol: $\delta_{i,j}=1$, if $i=j$, and $\delta_{i,j}=0$, if $i\neq j$. For any bivariate polynomial $P=P(x,y)$, we let $\deg(P)$ denote the degree of $P$, that is, the highest degree of the homogeneous terms of $P$ in both variables. We also let $\deg_1(P)$ (resp.\ $\deg_2(P)$) denote the degree of $P$ in its first (resp.\ second) variable. For any nonnegative integer $k\leq\deg(P)$, unless otherwise stated we let $P_k$ denote the homogeneous component of degree $k$ of $P$, that is, the polynomial obtained from $P$ by considering the terms of degree $k$ only. For any monomial $M$ of $P$, we let $[M]P(x,y)$ denote the coefficient of $M$ in $P(x,y)$ (we let $[M]P(x,y)=0$ if $M$ is not a monomial of $P$), and similarly for polynomials in more than two indeterminates. Finally, we define the following trivariate polynomial
$$
J_P(x,y,z) ~=~ P(P(x,y),z)+P(P(y,z),x)+P(P(z,x),y).
$$

Recall that the definition of $\calR$ enables us to identify the ring $\calR[x_1,\ldots,x_n]$ of polynomials in $n$ indeterminates over $\calR$ with the ring of polynomial functions from $\calR^n$ to $\calR$. Recall also that if $\chR=p>0$, then $p$ must be prime. In this case we have $(x+y)^p=x^p+y^p$ for any $x,y\in\calR$ and this identity (often referred to as the freshman's dream) immediately extends to any sum of more than two terms.

In this paper we will often make use of the following theorem, established in 1878 by E.~Lucas \cite{Luc78a,Luc78b,Luc78c}. For a more recent reference, see \cite{Fin47}.

\begin{theorem}[Lucas' theorem]
For any integers $n,m\geq 0$ and any prime $p\geq 2$, the following congruence relation holds:
$$
{n\choose m} ~\equiv ~\prod_{i=0}^k{n_i\choose m_i}{\quad}\modp,
$$
where $n=\sum_{i=0}^{k}n_ip^i$ and $m=\sum_{i=0}^{k}m_ip^i$ are the base $p$ expansions of $n$ and $m$, respectively. This uses the convention that ${a\choose b}=0$ for any integers $a,b$ such that $0\leq a<b$.
\end{theorem}

\begin{corollary}\label{cor:Sing}
For any integer $n>1$ and any prime $p$, the following two conditions are equivalent.
\begin{enumerate}
\item[(i)] $n\in s_1(p)$.
\item[(ii)] $p$ divides ${n\choose m}$ for any integer $m$ such that $0<m<n$.
\end{enumerate}
Moreover, if $\chR$ is a prime $p$, then any of these conditions holds iff $(x+y)^n=x^n+y^n$ for any $x,y\in\calR$.
\end{corollary}

\begin{proof}
(i) $\Rightarrow$ (ii). This implication immediately follows from Lucas' theorem.

(ii) $\Rightarrow$ (i). We prove this implication by contradiction. Suppose $n\notin s_1(p)$. Let $n=\sum_{i=0}^{k}n_ip^i$ be the base $p$ expansion of $n$, let $j\in\{0,\ldots,k\}$ such that $n_j\neq 0$, and let $m=n-p^j$. Then we have $0<m<n$ and by Lucas' theorem we also have ${n\choose m}\equiv n_j\modp$. This means that $p$ does not divide ${n\choose m}$, which is a contradiction.

\noindent The second part of the corollary is straightforward.
\end{proof}

\begin{corollary}\label{cor:Sing2}
Let $n>1$ be an integer and let $p$ be a prime.
\begin{enumerate}
\item[(a)] If $n=n_1+n_2\in s_2(p)$ for some $n_1,n_2\in s_1(p)$ (with $n_1,n_2$ distinct if $p=2$), then
    \begin{itemize}
    \item $p$ divides ${n\choose m}$ for any integer $m\in\{1,\ldots,n-1\}\setminus\{n_1,n_2\}$.
    \item ${n\choose n_1}\equiv {n\choose n_2}\equiv (1+\delta_{n_1,n_2})\modp$.
    \end{itemize}
\item[(b)] If $p$ divides ${n\choose m}{m\choose\ell}$ for any integers $\ell,m$ such that $0<\ell <m<n$, then $n\in s_1(p)\cup s_2(p)$.
\end{enumerate}
\end{corollary}

\begin{proof}
Assertion (a) is a straightforward consequence of Lucas' theorem. To show that assertion (b) holds, we first proceed as in the proof of the implication (ii) $\Rightarrow$ (i) of Corollary~\ref{cor:Sing}. Suppose $n\notin s_1(p)\cup s_2(p)$. Let $n=\sum_{i=0}^{k}n_ip^i$ be the base $p$ expansion of $n$, let $j\in\{0,\ldots,k\}$ such that $n_j\neq 0$, and let $m=n-p^j$. Then we have $0<m<n$ and ${n\choose m}\not\equiv 0\modp$. Since $n\notin s_2(p)$ we must have $m\notin s_1(p)$ and we conclude the proof by applying Corollary~\ref{cor:Sing}.
\end{proof}

We now prove the Main Theorem. Let $P\colon\calR^2\to\calR$ be a polynomial function satisfying Jacobi's identity \eqref{eq:Jacobi}, that is, such that $J_P=0$.

Suppose first that $\deg_2(P)=0$, that is, $P(x,y)=P(x)$. Using Jacobi's identity, we obtain that $P(P(x))$ is a constant, and hence $P$ is a constant $C$ satisfying $3C=0$. Therefore, $C$ can be any constant if $\chR=3$, and $C=0$, otherwise. Thus, we shall henceforth assume that $\deg_2(P)\geq 1$.

\begin{proposition}\label{prop:d1lt1}
If $P\colon\calR^2\to\calR$ is a polynomial function satisfying $J_P=0$ and $\deg_2(P)\geq 1$, then $\deg_1(P)\leq 1$.
\end{proposition}

We prove Proposition~\ref{prop:d1lt1} by contradiction. Thus we suppose that $\deg_1(P)=d\geq 2$.

\begin{claim}\label{claim:1}
We have $\deg_2(P)=\deg(P)=d$. Moreover, the polynomial function $P$ is of the form
\begin{equation}\label{eq:4s6d54}
P(x,y) ~=~ \sum_{k=0}^d\sum_{j=0}^k c_{k,j}{\,}x^jy^{k-j}
\end{equation}
with $c_{d,d}{\,}c_{d,0}\neq 0$, $c_{d,d}^d+c_{d,0}^d=0$, and
\begin{equation}\label{eq:prop5422}
P_d(x,y)^d ~=~ c_{d,d}^d(x^{d^2}-y^{d^2}).
\end{equation}
\end{claim}

\begin{proof}[Proof of Claim~\ref{claim:1}]
In this proof we use the notation $[x^k]_xJ_P(x,y,z)$ to denote the coefficient of $x^k$ in the expansion of $J_P$ in powers of $x$.

Set $d_2=\deg_2(P)\geq 1$. Then there exist polynomial functions $R_j\colon\calR\to\calR$ ($j=0,\ldots,d$) and $S_k\colon\calR\to\calR$ ($k=0,\ldots,d_2$), with $R_{d}\neq 0$ and $S_{d_2}\neq 0$, such that
$$
P(x,y) ~=~ \sum_{j=0}^{d}x^j R_j(y) ~=~ \sum_{k=0}^{d_2}y^k S_k(x).
$$
We then have
$$
J_P(x,y,z) ~=~ \sum_{j=0}^{d}\Big(\sum_{k=0}^{d}x^kR_k(y)\Big)^jR_j(z)+\sum_{k=0}^{d_2}x^kS_k(P(y,z))
+\sum_{j=0}^{d}\Big(\sum_{k=0}^{d_2}x^kS_k(z)\Big)^jR_j(y).
$$

Now, if $d>d_2$, then
$$
[x^{d^2}]_xJ_P(x,y,z) ~=~ R_{d}(y)^{d}{\,}R_{d}(z)
$$
from which we derive $R_{d}=0$, a contradiction. Similarly, if $d<d_2$, then
$$
[x^{d{\,}d_2}]_xJ_P(x,y,z) ~=~ S_{d_2}(z)^{d}{\,}R_{d}(y)
$$
and hence we obtain $R_{d}=0$ or $S_{d_2}=0$, again a contradiction. Thus we have proved that $d=d_2$. It then follows that
$$
[x^{d^2}]_xJ_P(x,y,z) ~=~ R_d(y)\big(R_d(y)^{d-1}R_d(z)+S_d(z)^d\big)
$$
and hence
\begin{equation}\label{eq:4sd56y}
R_d(y)^{d-1}R_d(z)+S_d(z)^d ~=~ 0.
\end{equation}
Since $d\geq 2$, from identity \eqref{eq:4sd56y} it follows that both $R_d$ and $S_d$ are nonzero constant polynomial functions. Thus the polynomial function $P$ is of the form
$$
P(x,y) ~=~ \sum_{j=0}^d\sum_{k=0}^d p_{j,k}{\,}x^jy^k,
$$
with $p_{d,0}{\,}p_{0,d}\neq 0$ and $p_{d,k}=p_{j,d}=0$ for $j,k=1,\ldots,d$. Identity \eqref{eq:4sd56y} also implies $p_{d,0}^d+p_{0,d}^d=0$.

Now, let $r=\deg(P)\geq d$ and let $M$ be an arbitrary monomial of $J_P$ of degree $rd$ in $(x,y)$ and degree $0$ in $z$ (e.g., $M=x^iy^{rd-i}$ for some $i\in\{0,\ldots,rd\}$). We then have
$$
[M]{\,}P(P(x,y),z) ~=~ [M]{\,}\sum_{j=0}^d p_{j,0}{\,}P(x,y)^j ~=~ [M]{\,}p_{d,0}{\,}P_r(x,y)^d
$$
and
\begin{eqnarray*}
\lefteqn{[M]{\,} P(P(y,z),x) ~=~ [M]{\,}\sum_{j=0}^d\sum_{k=0}^d p_{j,k}{\,}P(y,0)^jx^k}\\
&=& [M]{\,}\Big(p_{d,0}P(y,0)^d+p_{0,d}{\,}x^d+\sum_{j=0}^{d-1}\sum_{k=0}^{d-1} p_{j,k}{\,}P(y,0)^jx^k\Big) ~=~ [M]{\,}\delta_{r,d}{\,}p_{d,0}^{d+1}{\,}y^{d^2}
\end{eqnarray*}
Indeed, for $j,k=0,\ldots,d-1$, $P(y,0)^jx^k$ is of degree $jd+k\leq (d-1)d+(d-1)<d^2\leq rd$.

We show similarly that
$$
[M]{\,} P(P(z,x),y) ~=~ [M]{\,}\sum_{j=0}^d\sum_{k=0}^d p_{j,k}{\,}P(0,x)^jy^k ~=~ [M]{\,}\delta_{r,d}{\,}p_{d,0}{\,}p_{0,d}^d{\,}x^{d^2}.
$$

Let us now show by contradiction that $\deg(P)=d$. Suppose that $r>d$. By combining the latter three identities with $J_P=0$ we immediately obtain $0=[M]{\,}J_P(x,y,z)=[M]{\,}p_{d,0}{\,}P_r(x,y)^d$, a contradiction. We then have $\deg(P)=d$. Using the same three identities for $r=d$, we obtain
$$
P_d(x,y)^d+p_{d,0}^d{\,}y^{d^2}+p_{0,d}^d{\,}x^{d^2} = 0.
$$
Finally, since $\deg(P)=d$, the polynomial function $P$ must be of the form \eqref{eq:4s6d54}, with $c_{d,d}=p_{d,0}$ and $c_{d,0}=p_{0,d}$. Therefore the identities $c_{d,d}{\,}c_{d,0}\neq 0$, $c_{d,d}^d+c_{d,0}^d=0$, and \eqref{eq:prop5422} hold.
\end{proof}

We now show that $\chR$ must be a prime number. This shows that a contradiction is already reached if $\chR =0$, which then proves Proposition~\ref{prop:d1lt1} in this case.

\begin{claim}\label{claim:2}
The characteristic of $\calR$ is a prime $p$ and we have $d\in s_1(p)$. Moreover, we have $P_d(x,y)=c_{d,d}(x^d-y^d)$ and $(x+y)^d=x^d+y^d$ for any $x,y\in\calR$.
\end{claim}

\begin{proof}[Proof of Claim~\ref{claim:2}]
By Claim~\ref{claim:1} we have $\deg(P)=d$ and $[y^d]P(x,y)=c_{d,0}\neq 0$. Then we have
\begin{equation}\label{eq:g4g88}
P_d(x,y) ~=~ c_{d,d}{\,}x^d +\sum_{j=0}^r c_{d,j}{\,}x^jy^{d-j}
\end{equation}
for some integer $0\leq r\leq d-1$, with $c_{d,r}\neq 0$. Equation \eqref{eq:prop5422} then becomes
\begin{equation}\label{eq:pDchp}
c_{d,d}^d{\,}y^{d^2}+\sum_{k=0}^{d-1}{d\choose k}(c_{d,d}{\,}x^d)^k\Big(\sum_{j=0}^r c_{d,j}{\,}x^jy^{d-j}\Big)^{d-k} ~=~ 0.
\end{equation}
Clearly, the literal part of the monomial of highest degree in $x$ in the left-hand side of \eqref{eq:pDchp} is $x^{d(d-1)+r}y^{d-r}$. Indeed, it corresponds to the values $k=d-1$ and $j=r$ in the sums and therefore has the coefficient $d{\,}c_{d,d}^{d-1}c_{d,r}$. Since $c_{d,d}{\,}c_{d,r}\neq 0$, we must have $d{\,}1=0$ (here the symbol $1$ denotes the identity of $\calR$). If follows that the characteristic of $\calR$ should be a prime $p\geq 2$ that divides $d$.

We now show by contradiction that $d\in s_1(p)$. Suppose that $d\notin s_1(p)$. Then by Corollary~\ref{cor:Sing} we can let $m$ be the greatest $k\in\{1,\ldots,d-1\}$ such that ${d\choose k}\not\equiv 0\modp$. Equation~\eqref{eq:pDchp} then reduces to
\begin{equation}\label{eq:pDchp2}
c_{d,d}^d{\,}y^{d^2}+\sum_{k=0}^{m}{d\choose k}(c_{d,d}{\,}x^d)^k\Big(\sum_{j=0}^r c_{d,j}{\,}x^jy^{d-j}\Big)^{d-k} ~=~ 0.
\end{equation}
The literal part of the monomial of highest degree in $x$ in the left-hand side of \eqref{eq:pDchp2} is $x^{md+r(d-m)}{\,}y^{(d-r)(d-m)}$. It corresponds to the values $k=m$ and $j=r$ and therefore has the coefficient ${d\choose m}{\,}c_{d,d}^{m}{\,}c_{d,r}^{d-m}\neq 0$, which leads to a contradiction. Therefore $d\in s_1(p)$ and hence by Corollary~\ref{cor:Sing} we have $(x+y)^d=x^d+y^d$ for any $x,y\in\calR$.

Now, by Claim~\ref{claim:1} we have $(c_{d,d}+c_{d,0})^d=c_{d,d}^d+c_{d,0}^d=0$ and hence $c_{d,d}+c_{d,0}=0$. By Corollary~\ref{cor:Sing} the identity \eqref{eq:pDchp} then reduces to
$$
c_{d,d}^d{\,}y^{d^2}+\sum_{j=0}^r c_{d,j}^d{\,}x^{dj}y^{d(d-j)} ~=~ 0{\,},
$$
which implies $r=0$. Using \eqref{eq:g4g88} we finally obtain $P_d(x,y)=c_{d,d}(x^d-y^d)$.
\end{proof}

\begin{remark}
From now on we will often make an implicit use of Fermat's little theorem: if $m\in s_1(p)$ then $a^m\equiv a\modp$ for every integer $a$.
\end{remark}

We will now show (through Claims~\ref{claim:4}--\ref{claim:7}) that for every integer $k$ such that $1<k\leq d$ the polynomial function $P_k$ is of one of the following three forms.
\begin{itemize}
\item\textbf{Type 0}: $P_k=0$.
\item\textbf{Type 1}: $P_k\neq 0$, $k\in s_1(p)$, and
$$
P_k(x,y) ~=~ c_{k,k}(x^k-y^k).
$$
\item\textbf{Type 2}: $P_k\neq 0$, $k=k_1+k_2\in s_2(p)$, with $k_1\geq k_2$ and $k_1,k_2\in s_1(p)$, and
$$
P_k(x,y) ~=~ c_{k,k}{\,}x^k+\frac{c_{k,k_1}}{1+\delta_{k_1,k_2}}{\,}(x^{k_1}y^{k_2}+x^{k_2}y^{k_1})+c_{k,0}{\,}y^k.
$$
Note: This latter form simply means that $c_{k,j}=0$ whenever $j\notin\{k,k_1,k_2,0\}$ and that $c_{k,k_1}=c_{k,k_2}$.
\end{itemize}

For every real $r\geq 0$ and every $m\in\{0,1,2\}$ we let
$$
\calS_{m,r} ~=~ \{k~\text{integer}\mid r<k\leq d~\text{and $P_k$ is of type $m$}\}.
$$
It is clear that the sets $\calS_{0,r}$, $\calS_{1,r}$, $\calS_{2,r}$ are pairwise disjoint. Moreover, if $r\leq r'\leq d$, then we have $\calS_{m,r}\supseteq\calS_{m,r'}\supseteq\calS_{m,d}=\varnothing$.

By Claim~\ref{claim:2} we have $d=\sup\calS_{1,1}$. Regarding $\calS_{2,1}$ we have two cases to consider.
\begin{itemize}
\item If $\calS_{2,1}=\varnothing$, then we set $r_0=1$.
\item If $\calS_{2,1}\neq\varnothing$, then we set $q=q_1+q_2=\sup\calS_{2,1}$, with $q_1\geq q_2$ and $q_1,q_2\in s_1(p)$. We also set $r_0=q_1+\frac{q_2q}{d}$. We then have $1\leq q_1<r_0<q<d$. Note that $r_0$ is an integer iff $d$ divides $q_2q=q_2^2(1+q_1/q_2)$. But $p$ does not divide $(1+q_1/q_2)$ since $q_1/q_2\in s_1(p)$. Hence $r_0$ is an integer iff $d$ divides $q_2^2$, or equivalently, iff $d\leq q_2^2$. In this case we also have $d\leq q_1q_2$ and hence $d$ divides $q_1q_2$.
\end{itemize}

Note that if $\calS_{2,r}\neq\varnothing$ for some $r\geq 1$, then clearly $\calS_{2,1}\neq\varnothing$ and $r<q$.

\begin{remark}
In all the equations that we will now consider, some expressions are associated with polynomial functions $P_k$ for which $k\in\calS_{2,1}$ (e.g., expressions involving $q$, $q_1$, and $q_2$). The proofs corresponding to those equations show that these expressions are to be ignored when $\calS_{2,1}=\varnothing$.
\end{remark}

For every real $r\geq 1$ we set
$$
\displaylines{%
\alpha_r ~=~ \sum_{\textstyle{k\in\calS_{1,r}\setminus\{d\}\atop\ell\in\calS_{1,r},{\,}k\ell =rd}}c_{k,k}{\,}c_{\ell,\ell}^k{\,},\cr
\beta_r ~=~ \sum_{k\in\calS_{2,r}}\sum_{\textstyle{a,b\in\calS_{1,r}\atop a,b>q,{\,}ak_1+bk_2=rd}}c_{k,k}{\,}c_{a,a}^{k_1}{\,}c_{b,b}^{k_2}~,\qquad\quad \gamma_r ~=~ c_{r,r}{\,}c_{d,d}^r+\beta_r.
}
$$
If $r$ is an integer, then we easily see that $\alpha_r=0$ if $r\notin s_1(p)$, and $\beta_r=0$ if $r\notin s_1(p)\cup s_2(p)$. If $r=r_0$ is not an integer, then $rd=q_1d+q_1q_2+q_2^2\in s_3(p)$ and hence $\alpha_r=\beta_r=0$. Since $c_{r,r}$ is to be ignored in this case, we also have $\gamma_r=0$.

The proofs of the following two claims (Claims~\ref{claim:4} and \ref{claim:5}) are rather technical. For this reason we relegate them to Appendix~\ref{app:A}.

\begin{claim}\label{claim:4}
Let $r\in\{\lceil r_0\rceil,\ldots,d-1\}$ be such that $\{r+1,\ldots,d\}\subseteq\bigcup_{m=0}^2\calS_{m,r}$. If $i,u$ are integers such that $1\leq i<u<r$, then
\begin{equation}\label{eq:ceq1}
(-1)^{u-i}{u\choose i}{\,}c_{r,u}c_{d,d}^u+\delta_{r,r_0}\delta_{u,\frac{qq_2}{d}}{\,}\chi_{\{\frac{q_1q_2}{d},\frac{q_2^2}{d}\}}(i){\,}c_{q,q_1}^{q_2+1} ~=~ 0.
\end{equation}
Here, $\chi_{\{j,k\}}(i)=\max\{\delta_{i,j},\delta_{i,k}\}$.
\end{claim}

\begin{claim}\label{claim:5}
Let $r\in\{\lceil r_0\rceil,\ldots,d-1\}\cup\{r_0\}$ be such that $\{\lfloor r\rfloor+1,\ldots,d\}\subseteq\bigcup_{m=0}^2\calS_{m,r}$. Then the following two conditions hold.
\begin{itemize}
\item If either $r,u$ are integers such that $1\leq u<r$, or $r=r_0$ is not an integer and $u=\frac{q_2q}{d}$, then
\begin{eqnarray}
\lefteqn{c_{d,d}{\,}c_{r,u}^d+(-1)^{r-u}{r\choose u}\gamma_r}\nonumber\\
&& \null +\delta_{r,r_0}(1+\delta_{q_1,q_2})c_{q,q}{\,}[x^{du}y^{d(r-u)}](P_d(x,y)^{q_1}P_q(x,y)^{q_2})~=~0,\label{eq:ceq2}
\end{eqnarray}
where the first two summands are to be ignored when $r=r_0$ is not an integer.
\item If $r$ is an integer, then
\begin{eqnarray}
\lefteqn{c_{d,d}(c_{r,r}^d+c_{r,0}^d)+(1+(-1)^r)\gamma_r}\nonumber\\
&& \null +\delta_{r,r_0}(1+\delta_{q_1,q_2})c_{q,q_1}c_{d,d}^{q_1}(c_{q,q}-c_{q,0})^{q_2}-\delta_{r,1}c_{d,d}~=~0.\label{eq:ceq3}
\end{eqnarray}
\end{itemize}
\end{claim}

\begin{claim}\label{claim:6}
We have $\{\lfloor r_0\rfloor+1,\ldots,d\}\subseteq\bigcup_{m=0}^2\calS_{m,1}$.
\end{claim}

\begin{proof}[Proof of Claim~\ref{claim:6}]
We prove by decreasing induction that any integer $k\in\{\lfloor r_0\rfloor+1,\ldots,d\}$ is in $\bigcup_{m=0}^2\calS_{m,1}$. This is true for $k=d$ since $d\in\calS_{1,1}$. Suppose that the result holds for $k=r+1,\ldots,d$ for some integer $r$ such that $r_0<r<d$ and let us show that it holds for $k=r$. There are three mutually exclusive cases to consider.
\begin{itemize}
\item If $r\notin s_1(p)\cup s_2(p)$, then $\beta_r=0$ and by Corollary~\ref{cor:Sing2}(b) there exists integers $i_0,u_0$ satisfying $1\leq i_0<u_0<r$ such that ${r\choose u_0}{u_0\choose i_0}\not\equiv 0\modp$. Using \eqref{eq:ceq1} with $i=i_0$ and $u=u_0$ we immediately obtain $c_{r,u_0}=0$. Then, using \eqref{eq:ceq2} with $u=u_0$ we obtain $\gamma_r=0$, which implies $c_{r,r}=0$ (since $\beta_r=0$). Using again \eqref{eq:ceq2} we obtain that $c_{r,u}=0$ for every integer $u$ such that $1\leq u<r$. Finally, by \eqref{eq:ceq3} we obtain $c_{r,0}=0$ and hence $P_r=0$, that is, $r\in\calS_{0,1}$.
\item If $r\in s_1(p)$, then by Corollary~\ref{cor:Sing} we have ${r\choose u}\equiv 0\modp$ for every integer $u$ such that $1\leq u<r$. Using \eqref{eq:ceq2} we then obtain $c_{r,u}=0$ for every integer $u$ such that $1\leq u<r$. In \eqref{eq:ceq3} we have $(-1)^r\equiv -1\modp$ and hence $0=c_{r,r}^d+c_{r,0}^d=(c_{r,r}+c_{r,0})^d$. Therefore $P_r$ is of type 0 or 1, that is, $r\in\calS_{0,1}\cup\calS_{1,1}$.
\item If $r=r_1+r_2\in s_2(p)$, with $r_1\geq r_2$ and $r_1,r_2\in s_1(p)$, then by Corollary~\ref{cor:Sing2}(a) we have ${r\choose u}\equiv 0\modp$ for every integer $u\in\{1,\ldots,r-1\}\setminus\{r_1,r_2\}$. Using \eqref{eq:ceq2} we obtain $c_{r,u}=0$ for every integer $u\in\{1,\ldots,r-1\}\setminus\{r_1,r_2\}$. Now, if $r_1\neq r_2$, then using \eqref{eq:ceq2} for $u=r_1$ and then for $u=r_2$, we obtain $c_{r,r_1}=c_{r,r_2}$. Therefore, $P_r$ is of type 0 or 2, that is, $r\in\calS_{0,1}\cup\calS_{2,1}$.
\end{itemize}
This completes the proof of the claim.
\end{proof}

We now show that $\{2,\ldots,d\}\subseteq\calS_{0,1}\cup\calS_{1,1}$ (i.e., $P_k$ is of type $0$ or $1$ for $k=2,\ldots,d$).

\begin{claim}\label{claim:7}
We have $r_0=1$ (i.e., $\calS_{2,1}=\varnothing$).
\end{claim}

\begin{proof}[Proof of Claim~\ref{claim:7}]
We proceed by contradiction. Suppose that $r_0>1$, that is, $\calS_{2,1}\neq\varnothing$ and $r_0=q_1+\frac{q_2q}{d}$. Using \eqref{eq:ceq2} with $r=r_0$ and $u=u_0=\frac{q_2q}{d}$, we obtain
\begin{equation}\label{eq:4f4d}
c_{d,d}{\,}c_{r_0,u_0}^d+(-1)^{q_1}{r_0\choose u_0}\gamma_{r_0}-(1+\delta_{q_1,q_2}){\,}c_{d,d}^{q_1}{\,}c_{q,q}^{q_2+1}~=~ 0.
\end{equation}

Setting $r=q$ and $u=q_1$ in \eqref{eq:ceq2} and \eqref{eq:ceq3}, we obtain
\begin{equation}\label{eq:ceq2q}
c_{d,d}{\,}c_{q,q_1}^d ~=~ (1+\delta_{q_1,q_2}){\,}c_{q,q}{\,}c_{d,d}^q
\end{equation}
and
\begin{equation}\label{eq:ceq3q}
c_{d,d}(c_{q,q}^d+c_{q,0}^d)+2c_{q,q}{\,}c_{d,d}^q ~=~ 0.
\end{equation}
Indeed, $\delta_{q,r_0}=0$ and since $\calS_{2,q}=\varnothing$ we have $\beta_q=0$ and hence $\gamma_q=c_{q,q}c_{d,d}^q$. Moreover, by Corollary~\ref{cor:Sing2}(a) we have ${q\choose q_1}\equiv (1+\delta_{q_1,q_2})\modp$.

Now we have two cases to consider.
\begin{itemize}
\item If $r_0$ is not an integer, then the first two summands of \eqref{eq:4f4d} are to be ignored and hence we immediately derive $c_{q,q}=0$. Then from \eqref{eq:ceq2q} and \eqref{eq:ceq3q} we derive $c_{q,0}=c_{q,q_1}=0$, that is, $P_q=0$ (i.e., $q\in\calS_{0,1}$), a contradiction.
\item If $r_0$ is an integer (in which case $d$ divides both $q_1q_2$ and $q_2^2$), then using \eqref{eq:ceq1} with $r=r_0$, $u=u_0=\frac{q_2q}{d}$, and $i=\frac{q_1q_2}{d}$ (we note that ${u_0\choose i}\equiv (1+\delta_{q_1,q_2})\modp$ by Corollary~\ref{cor:Sing2}(a)) and then raising both sides of the resulting equation to the power $d$ we obtain
\begin{equation}\label{eq:ceq1rz}
c_{q,q_1}^{d(q_2+1)} ~=~ (1+\delta_{q_1,q_2}){\,}c_{r_0,u_0}^dc_{d,d}^{du_0}.
\end{equation}
Raising both sides of \eqref{eq:ceq2q} to the power $(q_2+1)$ and then combining the resulting equation with \eqref{eq:ceq1rz} we obtain
$$
c_{r_0,u_0}^d ~=~ (1+\delta_{q_1,q_2}){\,}c_{d,d}^{q_1-1}c_{q,q}^{q_2+1}.
$$
Substituting for $c_{r_0,u_0}^d$ into \eqref{eq:4f4d} and observing by Lucas' theorem that ${r_0\choose u_0}\equiv 1\modp$ we obtain $\gamma_{r_0}=0$.

Now, using \eqref{eq:ceq1} with $r=r_0$, $u=u_0=q_1+\frac{q_1q_2}{d}$, and $i=q_1$, we obtain ${u_0\choose q_1}{\,}c_{r_0,u_0}c_{d,d}^{u_0}=0$, and therefore $c_{r_0,u_0}=0$ (since ${u_0\choose q_1}\equiv 1\modp$ by Corollary~\ref{cor:Sing2}(a)). Using \eqref{eq:ceq2} with the same $r=r_0$ and $u=u_0$, we then obtain
$$
c_{q,q}{\,}[x^{q_1d+q_1q_2}y^{q_2^2}](P_d(x,y)^{q_1}P_q(x,y)^{q_2})~=~0,
$$
that is, $c_{q,q}{\,}c_{q,q_1}^{q_2}c_{d,d}^{q_1}=0$. Combining this latter equation with \eqref{eq:ceq2q} and \eqref{eq:ceq3q} we obtain $c_{q,q}=c_{q,0}=c_{q,q_1}=0$, that is, $P_q=0$, a contradiction.
\end{itemize}
This completes the proof of the claim.
\end{proof}

\begin{proof}[Proof of Proposition~\ref{prop:d1lt1}]
On the one hand, using \eqref{eq:ceq3} with $r=r_0=1$ and the fact that $\calS_{2,1}=\varnothing$ (i.e., $q$ does not exist), we obtain
$$
0 ~=~ c_{d,d}(c_{1,1}^d+c_{1,0}^d-1^d) ~=~ c_{d,d}(c_{1,1}+c_{1,0}-1)^d,
$$
that is,
\begin{equation}\label{eq:comp1}
c_{1,1}+c_{1,0} ~=~ 1.
\end{equation}
On the other hand, by Claims~\ref{claim:6} and \ref{claim:7} for any $M\in\{x,y,z\}$ we have
\begin{eqnarray*}
[M]P(P(x,y),z) &=& [M]\Big(\sum_{k\in\calS_{1,1}}P_k(P(x,y),z)+P_1(P(x,y),z)+P_0\Big)\\
&=& [M]\Big(\sum_{k\in\calS_{1,1}}c_{k,k}(P(x,y)^k-z^k)+c_{1,1}P(x,y)+c_{1,0}z+c_{0,0}\Big).
\end{eqnarray*}
Clearly, the sum over $k\in\calS_{1,1}$ above cannot contain monomials of degree $1$. Therefore we have
$$
[M]P(P(x,y),z) ~=~ [M](c_{1,1}P_1(x,y)+c_{1,0}z) ~=~ [M](c_{1,1}(c_{1,1}x+c_{1,0}y)+c_{1,0}z).
$$
Since the identity $[x]J_P(x,y,z)=0$ can be written as $\sum_{M\in\{x,y,z\}}[M]P(P(x,y),z)=0$, we have
\begin{equation}\label{eq:comp2}
c_{1,1}(c_{1,1}+c_{1,0})+c_{1,0} ~=~ 0.
\end{equation}
Since the system \eqref{eq:comp1}--\eqref{eq:comp2} is inconsistent we immediately reach a contradiction.
\end{proof}

\begin{proof}[Proof of the Main Theorem]
By Proposition~\ref{prop:d1lt1}, there exist two polynomial functions $R\colon\calR\to\calR$ and $S\colon\calR\to\calR$ such that
$$
P(x,y) ~=~ x{\,}R(y)+S(y).
$$
We then have
\begin{eqnarray*}
J_P(x,y,z) &=& x{\,}R(y)R(z)+S(y)R(z)+S(z)\\
&& \null + y{\,}R(z)R(x)+S(z)R(x)+S(x)\\
&& \null + z{\,}R(x)R(y)+S(x)R(y)+S(y).
\end{eqnarray*}
Suppose that $\deg(R)=r>1$ and set $A=[y^r]R(y)$. We can then readily see that
$$
[x{\,}y^rz^r]J_P(x,y,z) ~=~ A^2.
$$
We then have $A=0$, a contradiction. Therefore $R(y)=A_1y+A_0$ for some $A_1,A_0\in\calR$. Now, suppose that $\deg(S)=s>1$ and set $B=[y^s]S(y)$. It is then easy to see that
$$
[y^s]J_P(x,y,z) ~=~ (A_0+1)B\quad\text{and}\quad [y^sz]J_P(x,y,z) ~=~ A_1B.
$$
However, one can readily see that $P$ cannot satisfy Jacobi's identity if $A_1=0$ and $A_0=-1$. Thus we must have $B=0$, again a contradiction.

Finally, the polynomial $P$ must be of the form
$$
P(x,y) ~=~ Axy+Bx+Cy+D
$$
for some $A,B,C,D\in\calR$ and we can immediately verify that this polynomial satisfies Jacobi's identity iff
$$
3A^2 ~=~ 3D(B+1) ~=~ A(2B+C) ~=~ B^2+BC+C+AD ~=~ 0.
$$
The statement of the Main Theorem then follows straightforwardly.
\end{proof}

\appendix
\section{Proofs of Claims \ref{claim:4} and \ref{claim:5}}\label{app:A}

Before providing the proofs of Claims \ref{claim:4} and \ref{claim:5}, we first show that for any $r\geq r_0$ and any $k=k_1+k_2\in\calS_{2,r}$, with $k_1\geq k_2$ and $k_1,k_2\in s_1(p)$, the following conditions hold.
\begin{enumerate}
\item[(a)] $k_1=q_1$ and $k_2\leq q_2$.
\item[(b)] $d(r-k_1)\geq qk_2$. The equality holds iff $r=r_0$ and $k=q$.
\item[(c)] $d(r-k_2)\geq qk_1$. The equality holds iff $r=r_0$, $k=q$, and $q_1=q_2$.
\item[(d)] $ak_1+bk_2\leq rd$ for all $a\leq d$ and $b\leq q$. The equality holds iff $a=d$, $b=q$, $k_2=q_2$, and $r=r_0$.
\item[(e)] $ak_1+bk_2\leq rd$ for all $a\leq q$ and $b\leq d$. The equality holds iff $a=q$, $b=d$, $k_2=q_2$, $q_1=q_2$, and $r=r_0$.
\end{enumerate}

\begin{proof}
if $\calS_{2,1}=\varnothing$, then $\calS_{2,r}=\varnothing$ for every $r\geq r_0=1$ and then there is nothing to prove. We therefore assume that $\calS_{2,1}\neq\varnothing$. We then have $r_0=q_1+q_2q/d$ and $q_1<r_0\leq r<k\leq q$.
\begin{enumerate}
\item[(a)] We have $k_1=q_1$. Indeed, if we had $k_1>q_1$, then we would have $k>k_1\geq pq_1\geq 2q_1\geq q_1+q_2=q$, a contradiction. If we had $k_1<q_1$, then we would have $q_1\geq pk_1\geq 2k_1\geq k_1+k_2=k$, a contradiction. Finally, $k\leq q$ implies $k_2\leq q_2$.
\item[(b)] We have $d(r-k_1)-qk_2\geq d(r_0-q_1)-qq_2=0$.
\item[(c)] We have $d(r-k_2)-qk_1\geq d(r_0-q_2)-qq_1=(q_1-q_2)(d-q)\geq 0$.
\item[(d)] We have $ak_1+bk_2\leq dq_1+qq_2=r_0d\leq rd$.
\item[(e)] We have $ak_1+bk_2\leq qq_1+dq_2\leq dq_1+qq_2=r_0d\leq rd$, where the second inequality is equivalent to $(q_1-q_2)(d-q)\geq 0$.\qedhere
\end{enumerate}
\end{proof}

\begin{proof}[Proof of Claim~\ref{claim:4}]
We consider the identity $[M]J_P(x,y,z)=0$ for $M=x^{di}y^{d(u-i)}z^{r-u}$. Since $di\geq d$ and $d(u-i)\geq d$, we have $[M]P(P(y,z),x)=0$ and $[M]P(P(z,x),y)=0$. Also, we have
\begin{eqnarray*}
[M]P(P(x,y),z) &=& [M]\sum_{k\in\calS_{1,r}}P_k(P(x,y),z)+[M]\sum_{k\in\calS_{2,r}}P_k(P(x,y),z)\\
&& \null + [M]P_r(P(x,y),z) + [M]\sum_{k<r}P_k(P(x,y),z).
\end{eqnarray*}
Let us compute the latter four summands separately.
\begin{itemize}
\item We clearly have
$$
[M]\sum_{k\in\calS_{1,r}}P_k(P(x,y),z) ~=~ [M]\sum_{k\in\calS_{1,r}}c_{k,k}(P(x,y)^k-z^k) ~=~0.
$$
\item Assuming that $\calS_{2,r}\neq\varnothing$ and setting $M'=x^{di}y^{d(u-i)}$, we obtain
\begin{eqnarray*}
\lefteqn{[M]\sum_{k\in\calS_{2,r}}P_k(P(x,y),z)}\\
&=& [M]\sum_{k\in\calS_{2,r}}\frac{c_{k,k_1}}{1+\delta_{k_1,k_2}}\big(P_{du/k_1}(x,y)^{k_1}z^{k_2}+P_{du/k_2}(x,y)^{k_2}z^{k_1}\big)\\
&=& [M']\sum_{k\in\calS_{2,r}}\frac{c_{k,k_1}}{1+\delta_{k_1,k_2}}\big(P_{du/k_1}(x,y)^{k_1}\delta_{k_2,r-u}+P_{du/k_2}(x,y)^{k_2}\delta_{k_1,r-u}\big)\\
&=& [M']\sum_{k\in\calS_{2,r}}\frac{c_{k,k_1}}{1+\delta_{k_1,k_2}}\big(P_{d(r-k_2)/k_1}(x,y)^{k_1}\delta_{k_2,r-u}+P_{d(r-k_1)/k_2}(x,y)^{k_2}\delta_{k_1,r-u}\big).
\end{eqnarray*}
If $d(r-k_2)/k_1>q$, then $P_{d(r-k_2)/k_1}$ is of type 0 or 1, so it does not contain any product terms and hence $M'$ cannot appear in $P_{d(r-k_2)/k_1}(x,y)^{k_1}$. We arrive at the same conclusion for $P_{d(r-k_1)/k_2}$. Using conditions (b) and (c) above, we then obtain
\begin{eqnarray*}
\lefteqn{[M]\sum_{k\in\calS_{2,r}}P_k(P(x,y),z)}\\
&=& [M']\frac{\delta_{r,r_0}{\,}c_{q,q_1}}{1+\delta_{q_1,q_2}}\big(\delta_{q_1,q_2}{\,}P_q(x,y)^{q_1}\delta_{q_2,r_0-u}+P_q(x,y)^{q_2}\delta_{q_1,r_0-u}\big)\\
&=& [M']{\,}\delta_{r,r_0}\delta_{u,\frac{qq_2}{d}}{\,}c_{q,q_1}P_q(x,y)^{q_2}\\
&=& [M']{\,}\delta_{r,r_0}\delta_{u,\frac{qq_2}{d}}{\,}\frac{c_{q,q_1}^{q_2+1}}{1+\delta_{q_1,q_2}}\big(x^{q_1q_2}y^{q_2^2}+x^{q_2^2}y^{q_1q_2}\big)\\
&=& \delta_{r,r_0}\delta_{u,\frac{qq_2}{d}}{\,}\chi_{\{\frac{q_1q_2}{d},\frac{q_2^2}{d}\}}(i){\,}c_{q,q_1}^{q_2+1}.
\end{eqnarray*}
\item Since $M$ is of degree $du$ in $(x,y)$, we have
\begin{eqnarray*}
[M]P_r(P(x,y),z) &=& [M]\sum_{j=0}^rc_{r,j}{\,}P(x,y)^jz^{r-j} ~=~ [M]{\,}c_{r,u}{\,}P(x,y)^uz^{r-u}\\
&=& [M]{\,}c_{r,u}{\,}c_{d,d}^u(x^d-y^d)^uz^{r-u} ~=~ (-1)^{u-i}{u\choose i}{\,}c_{r,u}c_{d,d}^u.
\end{eqnarray*}
\item Let us now compute $[M]\sum_{k<r}P_k(P(x,y),z)$. If $k<r-u$, the degree in $z$ of $P_k(P(x,y),z)$ cannot reach $r-u$. If $r-u\leq k<r$, we have
$$
[M]P_k(P(x,y),z) ~=~ [M]{\,}c_{k,k-(r-u)}P(x,y)^{k-(r-u)}z^{r-u}.
$$
This expression is $0$ since the degree in $(x,y)$ of $P(x,y)^{k-(r-u)}$ does not exceed $d(k-(r-u))<du$.
\end{itemize}
This completes the proof of the claim.
\end{proof}

\begin{proof}[Proof of Claim~\ref{claim:5}]
We first consider the identity $[M]J_P(x,y,z)=0$ for the monomials $M=x^{du}y^{d(r-u)}$ with $0\leq u\leq r$. These monomials are of degree $rd$ in $(x,y)$ and $0$ in $z$. Thus we have
\begin{eqnarray*}
[M]P(P(x,y),z) &=& [M]\sum_{k\in\calS_{1,r}}c_{k,k}P(x,y)^k+[M]\sum_{k\in\calS_{2,r}}c_{k,k}P(x,y)^k\\
&& \null + [M]{\,}c_{r,r}P(x,y)^r + [M]\sum_{k<r}c_{k,k}P(x,y)^k.
\end{eqnarray*}
Let us compute the latter four summands separately.
\begin{itemize}
\item We show that $[M]\sum_{k\in\calS_{1,r}}c_{k,k}P(x,y)^k = [M](c_{d,d}P_r(x,y)^d+\alpha_r(x-y)^{rd})$. Since $k\in\calS_{1,r}$ implies $k\in s_1(p)$, we have
    $$
    [M]\sum_{k\in\calS_{1,r}}c_{k,k}P(x,y)^k ~=~ [M]\sum_{k\in\calS_{1,r}}c_{k,k}P_{\frac{rd}{k}}(x,y)^k.
    $$
We then observe that if $d>k\in\calS_{1,r}$ (hence $d\geq pk$) and $P_{\ell}\neq 0$, with $\ell=\frac{rd}{k}$, then necessarily $\ell\in\calS_{1,r}$. Indeed, since $\ell >r$ we must have $\ell\in\calS_{1,r}\cup\calS_{2,r}$ by the hypotheses of the claim. If $r_0=1$, then $\calS_{2,r}=\varnothing$ and hence $\ell\in\calS_{1,r}$. If $r_0>1$, then we have $\ell=\frac{rd}{k}\geq pr\geq 2r_0>2q_1\geq q$, and hence $\ell\in\calS_{1,r}$ by definition of $q$.

Therefore, we have
\begin{eqnarray*}
[M]\sum_{k\in\calS_{1,r}\setminus\{d\}}c_{k,k}P(x,y)^k &=& [M]\sum_{k\in\calS_{1,r}\setminus\{d\}}c_{k,k}\sum_{\ell\in\calS_{1,r},{\,}k\ell=rd}c_{\ell,\ell}^k(x-y)^{rd}\\ &=& [M]{\,}\alpha_r(x-y)^{rd},
\end{eqnarray*}
which immediately gives the stated identity.
\item Assuming that $\calS_{2,r}\neq\varnothing$, let us show that
$$
[M]\sum_{k\in\calS_{2,r}}c_{k,k}P(x,y)^k ~=~ [M]\big(\beta_r(x-y)^{rd}+\delta_{r,r_0}(1+\delta_{q_1,q_2})c_{q,q}(P_d(x,y)^{q_1}P_q(x,y)^{q_2})\big).
$$
Indeed, the left-hand side of this identity can be rewritten as
$$
[M]\sum_{k\in\calS_{2,r}}\sum_{a,b=0}^dc_{k,k}P_a(x,y)^{k_1}P_b(x,y)^{k_2}.
$$
Since $P_a(x,y)^{k_1}P_b(x,y)^{k_2}$ is a homogeneous polynomial function of degree $ak_1+bk_2$, we can use conditions (d) and (e) above to analyze all the summands corresponding to $a\leq q$ or $b\leq q$. If $a>q$ and $b>q$ (hence $a>r$ and $b>r$ since $q>r$ when $\calS_{2,r}\neq\varnothing$), then necessarily $a,b\in\calS_{0,r}\cup\calS_{1,r}$ and we obtain the stated identity.
\item We have $[M]{\,}c_{r,r}P(x,y)^r = [M]{\,}c_{r,r}P_d(x,y)^r = [M]{\,}c_{r,r}c_{d,d}^r(x-y)^{rd}$.
\item We have $[M]\sum_{k<r}c_{k,k}P(x,y)^k=0$ since the degree of $P(x,y)^k$ is bounded by $kd<rd$.
\end{itemize}
Summing up, we obtain
\begin{eqnarray}
[M]P(P(x,y),z) &=& [M]\big(c_{d,d}P_r(x,y)^d+(\alpha_r+\gamma_r)(x-y)^{rd}\nonumber\\
&& \null + \delta_{r,r_0}(1+\delta_{q_1,q_2})c_{q,q}(P_d(x,y)^{q_1}P_q(x,y)^{q_2})\big).\label{eq:s35h}
\end{eqnarray}

If $r,u$ are integers such that $1\leq u<r$, then $M=x^{du}y^{d(r-u)}$ is a polynomial multiple of $x^dy^d$. Since no monomial in $P(P(y,z),x)$ and $P(P(z,x),y)$ is a polynomial multiple of $x^dy^d$ we must have $[M]J_P(x,y,z)=[M]P(P(x,y),0)$. We then observe that if $\alpha_r\neq 0$, then $r\in s_1(p)$ and in this case we have $[M](x-y)^{rd}=0$ and hence $\alpha_r$ can be ignored in \eqref{eq:s35h}. We then immediately obtain \eqref{eq:ceq2}.

If $r=r_0$ is not an integer and $u=\frac{q_2q}{d}$, then $M=x^{q_2q}y^{q_1d}$ and $r<q$. We then have
$$
[M]P(P(y,z),x) ~=~ [M]\sum_{k\leq q}P_k(P(y,z),x)+[M]\sum_{k>q}P_k(P(y,z),x),
$$
where the first summand is clearly zero. The second summand is also zero since $k>q>r$ implies $k\in\calS_{0,r}\cup\calS_{1,r}$. We show similarly that $[M]P(P(z,x),y)=0$. Moreover, the summands involving $P_r$ and $(\alpha_r+\gamma_r)$ are to be ignored in \eqref{eq:s35h}. We therefore obtain \eqref{eq:ceq2}, in which the first two summands are to be ignored.

Let us now prove \eqref{eq:ceq3}. We consider the monomial $M=x^{rd}$ and hence we have
$$
[M]J_P(x,y,z) ~=~ [M]P(P(x,0),0)+[M]P(P(0,x),0)+[M]P(P(0,0),x).
$$
The first summand is exactly the right-hand side of \eqref{eq:s35h} when $u=r$, that is
$$
c_{d,d}{\,}c_{r,r}^d+(\alpha_r+\gamma_r)+ \delta_{r,r_0}(1+\delta_{q_1,q_2}){\,}c_{q,q}{\,}c_{d,d}^{q_1}{\,}c_{q,q}^{q_2}.
$$
Similarly, the second summand is the right-hand side of \eqref{eq:s35h} when $u=0$, that is
$$
c_{d,d}{\,}c_{r,0}^d+(\alpha_r+\gamma_r)(-1)^r- \delta_{r,r_0}(1+\delta_{q_1,q_2}){\,}c_{q,q}{\,}c_{d,d}^{q_1}{\,}c_{q,0}^{q_2}.
$$
The third summand is simply equal to $\delta_{r,1}c_{d,0}=-\delta_{r,1}c_{d,d}$ since $d\in\calS_{1,r}$. We then conclude the proof by observing that $(1+(-1)^r){\,}\alpha_r=0$ since if $\alpha_r\neq 0$ then $r\in s_1(p)$.
\end{proof}

\section{Case of equations~\eqref{eq:J3} and \eqref{eq:J4}}\label{app:B}

The functional equations corresponding to \eqref{eq:J3} and \eqref{eq:J4} are respectively given by
\begin{eqnarray}
P(P(x,y),z)+P(y,P(x,z))-P(x,P(y,z)) &=& 0,\label{eq:J5}\\
P(x,P(y,z))+P(P(x,z),y)-P(P(x,y),z) &=& 0.\label{eq:J6}
\end{eqnarray}
It is then easy to see that $P$ satisfies \eqref{eq:J6} iff the polynomial $P'$ defined by $P'(x,y)=P(y,x)$ satisfies \eqref{eq:J5}.

Now, let $P\colon\calR^2\to\calR$ be a polynomial function satisfying \eqref{eq:J5} and let us show that necessarily $P=0$.

Suppose that $\deg_2(P)\geq 1$ and let us prove by contradiction that $\deg_1(P)\leq 1$. Suppose that $\deg_1(P)=d\geq 2$. By using the notation of the proof of Claim~\ref{claim:1}, we see that \eqref{eq:J5} can be rewritten as
$$
\sum_{j=0}^d\Big(\sum_{k=0}^dx^kR_k(y)\Big)^jR_j(z) + \sum_{k=0}^{d_2}\Big(\sum_{j=0}^dx^jR_j(z)\Big)^kS_k(y)-\sum_{j=0}^dx^jR_j(P(y,z))~=~0.
$$
If $d>d_2$ (resp.\ $d<d_2$), then by equating the coefficients of $x^{d^2}$ (resp.\ $x^{dd_2}$) in the expansion in powers of $x$ of each side of the latter equation, we obtain a contradiction. Therefore, we have $d=d_2$. By equating the coefficients of $x^{d^2}$ we then obtain
$$
R_d(y)^d+R_d(z)^{d-1}S_d(y) ~=~ 0,
$$
which shows that both $R_d$ and $S_d$ are nonzero constant polynomial functions.

Now, by identifying $x$ and $y$ in \eqref{eq:J5}, we obtain
\begin{equation}\label{eq:diag33}
P(P(x,x),z) ~=~ 0,
\end{equation}
or equivalently,
$$
\sum_{k=0}^dz^kS_k(P(x,x)) ~=~ 0.
$$
By equating the coefficients of $z^d$ in the latter equation we obtain $S_d=0$, a contradiction. Therefore we have $\deg_1(P)\leq 1$ and hence we have
$$
P(x,y) ~=~ x{\,}R_1(y)+R_0(y).
$$
Substituting in \eqref{eq:diag33}, we then obtain
$$
(x{\,}R_1(x)+R_0(x)){\,}R_1(z)+R_0(z) ~=~ 0.
$$
If $x{\,}R_1(x)+R_0(x)$ is nonconstant, then $R_1=0$ and then also $R_0=0$. Otherwise, if $x{\,}R_1(x)+R_0(x)$ is a constant $C$, then $R_0(z)=-C{\,}R_1(z)$ and hence $C=x{\,}R_1(x)+R_0(x)=x{\,}R_1(x)-C{\,}R_1(x)$, from which we derive $R_1=0$ and then also $R_0=0$. Finally, $P=0$, which contradicts the assumption that $\deg_2(P)\geq 1$. Hence we have $\deg_2(P)=0$, in which case we immediately see that $P=0$.

\section*{Acknowledgments}

This research is partly supported by the internal research project R-AGR-0500 of the University of Luxembourg. The authors thank Michel Rigo of the University of Li\`ege for pointing out Lucas' theorem. They also thank J\"org Tomaschek of Deloitte Austria for bringing this problem to their attention.


\begin{thebibliography}{99}

\bibitem{Bok73}
O.~G.~Bokov.
\newblock A model of Lie fields and multiple-time retarded Green’s functions of an electromagnetic field in dielectric media.
\newblock {\em Nauchn. Tr. Novosib. Gos. Pedagog. Inst.} 86:3--9, 1973.

\bibitem{Fin47}
N.~Fine.
\newblock Binomial coefficients modulo a prime.
\newblock {\em Amer. Math. Monthly} 54:589--592, 1947.

\bibitem{Fri16}
H.~Fripertinger.
\newblock On $n$-associative formal power series.
\newblock {\em Aeq. Math.} 90(2):449--467, 2016.

\bibitem{Gil74}
R.~Gilmore.
\newblock {\em Lie groups, Lie algebras, and some of their applications}.
\newblock John Wiley and Sons, New York, 1974.

\bibitem{Hal03}
B.~C.~Hall.
\newblock {\em Lie groups, Lie algebras, and representations. An elementary introduction}.
\newblock Springer-Verlag, New York, 2003.

\bibitem{Jac79}
N.~Jacobson.
\newblock {\em Lie algebras}.
\newblock Courier Dover Publications, 1979.

\bibitem{Luc78a}
E.~Lucas.
\newblock Th\'eorie des fonctions num\'eriques simplement p\'eriodiques.
\newblock {\em Am. J. Math.} 1(2): 184--196, 1878.

\bibitem{Luc78b}
E.~Lucas.
\newblock Th\'eorie des fonctions num\'eriques simplement p\'eriodiques.
\newblock {\em Am. J. Math.} 1(3): 197--240, 1878.

\bibitem{Luc78c}
E.~Lucas.
\newblock Th\'eorie des fonctions num\'eriques simplement p\'eriodiques.
\newblock {\em Am. J. Math.} 1(4): 289--321, 1878.

\bibitem{MarMat11}
J.-L.~Marichal and P.~Mathonet.
\newblock A description of n-ary semigroups polynomial-derived from integral domains.
\newblock {\em Semigroup Forum} 83:241--249, 2011


\bibitem{Toma}
J\"org Tomaschek. Deloitte Austria. Private communication.

\bibitem{Yag82}
A.~V.~Yagzhev.
\newblock A functional equation from theoretical physics.
\newblock {\em Funct. Anal. Appl.} 16(1):38--44, 1982.
\end{thebibliography}
\end{document}